\documentclass[reqno,11pt]{amsart}
\usepackage{amsmath,amssymb,amsfonts,amsthm, amscd,indentfirst}
\usepackage{amsmath,latexsym,amssymb,amsmath,
	amscd,amsthm,amsxtra}

\usepackage[hyphens]{url}
\usepackage[colorlinks=true,linkcolor=blue,citecolor=blue,urlcolor=blue,hypertexnames=false,linktocpage]{hyperref}
\usepackage{bookmark}
\usepackage{amsmath,thmtools,mathtools}
\mathtoolsset{showonlyrefs=true}
\usepackage[T1]{fontenc}
\usepackage{newtxtext,newtxmath}

\usepackage[msc-links]{amsrefs} 

\newcounter{mparcnt}

\usepackage{fancyhdr}
\usepackage{esint}
\usepackage{enumerate}
\usepackage{xcolor}

\usepackage{pictexwd,dcpic}
\usepackage{graphicx}
\usepackage{caption}
\usepackage{graphicx}
\usepackage{caption}
\usepackage{epsfig,here}
\usepackage{subfigure,here}

\newtheorem{theorem}{Theorem}[section]
\newtheorem{lemma}[theorem]{Lemma}
\newtheorem{proposition}[theorem]{Proposition}
\newtheorem{remark}[theorem]{Remark}

\def\Om{\Omega}
\def\p{\partial}

\def\De{\Delta}

\def\S{{\Sigma}}
\def\<{\langle}
\def\>{\rangle}
\def\div{{\rm div}}
\def\na{\nabla}

\providecommand{\abs}[1]{\lvert#1\rvert}

\newcommand{\mbB}{\mathbb{B}}

\newcommand{\mbR}{\mathbb{R}}

\newcommand{\mfR}{\mathbf{R}}

\newcommand{\rd}{{\rm d}}

\newcommand{\eq}[1]{\begin{equation}\begin{alignedat}{2} #1 \end{alignedat}\end{equation}}

\numberwithin{equation} {section}

\begin{document}
	
	\title[A characterization of capillary spherical cap]{A characterization of capillary spherical caps by a partially overdetermined problem in a half ball}
\author[Jia]{Xiaohan Jia}
\address[X.J]{School of Mathematics\\
Southeast University\\
 211189, Nanjing, P.R. China}
\email{jiaxiaohan@xmu.edu.cn}
\author[Lu]{Zheng Lu}
\address[Z.L]{School of Mathematical Sciences\\
	Xiamen University\\
	361005, Xiamen, P.R. China
	\newline\indent Mathematisches Institut \\
	Albert-Ludwigs-Universität Freiburg\\
	Freiburg im Breisgau, 79104, Germany}
\email{zhenglu@stu.xmu.edu.cn}
\author[Xia]{Chao Xia}
\address[C.X]{School of Mathematical Sciences\\
	Xiamen University\\
	361005, Xiamen, P.R. China}
\email{chaoxia@xmu.edu.cn}
\author[Zhang]{Xuwen Zhang}
\address[X.Z]{School of Mathematical Sciences\\
	Xiamen University\\
	361005, Xiamen, P.R. China
	\newline\indent Institut f\"ur Mathematik, Goethe-Universit\"at, 
	60325, Frankfurt, Germany
}
\curraddr{Mathematisches Institut\\
		Universit\"at Freiburg\\
	Ernst-Zermelo-Str.1\\
		79104\\
  \newline\indent Freiburg\\ Germany}
\email{xuwen.zhang@math.uni-freiburg.de}
	
\begin{abstract}
In this note, we study a Serrin-type partially overdetermined problem proposed by Guo-Xia (Calc. Var. Partial Differential Equations 58: no. 160, 2019. \href{https://doi.org/10.1007/s00526-019-1603-3}{https://doi.org/10.1007/s00526-019-1603-3}), and prove a rigidity result that characterizes capillary spherical caps in a half ball. 
		
		\
		
		
\end{abstract}

\maketitle

\section{Introduction}\label{Sec-1}
In a celebrated article \cite{Serrin71}, Serrin proposed the overdetermined boundary value problem in a bounded domain $\Omega\subset\mbR^{n+1}$:
\eq{\label{defn-BVP-closed}
\begin{cases}
	\bar{\Delta} f=1&\text{in }\Omega,\\
	f=0&\text{on } \partial \Omega,\\
	\bar{\nabla}_{v}f=c&\text{on } \partial \Omega,
\end{cases}
}
where $\p\Om$ is a smooth hypersurface, $c\in\mbR$ is a constant and $\nu$ is the outer unit normal to $\p\Om$.
Using Alexandrov’s moving plane method, Serrin showed that, \eqref{defn-BVP-closed} admits a solution if and only if $\partial\Omega$ is a round sphere and $f$ is radially symmetric.
An alternative approach to Serrin's symmetric theorem is the so-called integral method,
initiated by
Weinberger in \cite{Weinberger71} and intensively developed by Magnanini-Poggesi in \cite{MP20-1,MP20-2}.
Precisely, for the solution to \eqref{defn-BVP-closed}, Magnanini-Poggesi have shown the following integral identity:
\eq{
\int_\Om(-f)\left\{\abs{\bar\na^2f}^2-\frac{(\bar\De f)^2}{n+1}\right\}\rd x
=\frac12\int_{\p\Om}(f_\nu^2-R^2)(f_\nu-g_\nu)\rd A,
}
where $f_\nu\coloneqq\bar\na_\nu f$, $R$ is any constant, and $g$ is any
quadratic distance function of the form $g(x)=\frac1{2(n+1)}(\abs{x-z}^2-a)$.
Because the term $\abs{\bar\na^2f}^2-\frac{(\bar\De f)^2}{n+1}$ on the left is exactly the trace-free Hessian of $f$, after properly choosing the value of $R$, one can show that $\bar\na^2f$ is umbilical, together with the fact $\bar\De f=1$ on $\Om$, implying that $f$ has to be quadratic, which in turn shows that $\p\Om$ has to be a round sphere due to the Dirichlet boundary condition.  
As one can easily see from this argument, 
Serrin's symmetry result has a very close relation with the rigidity of embedded closed constant mean curvature (CMC) hypersurfaces — the groundbreaking Alexandrov's soap bubble theorem.
We refer to a nice survey paper by Magnanini \cite{Magnanini17} regarding this insight.


In this paper, we are interested in the study of Serrin-type symmetry results in a half ball, which is closely related to the study of the Alexandrov-type theorem for capillary CMC hypersurfaces—hypersurfaces in a given container with constant mean curvature and intersecting the boundary of the container at a constant contact angle $\theta\in (0,\pi)$.
These are natural geometric objects to be studied, from the calculus and variational point of view, since they appear naturally as the stationary points of the free energy functional under volume constraint.
For the background we refer to the beautiful monographs by Finn \cite{Finn86} and Maggi \cite{Mag12}.
Alexandrov-type theorem for capillary CMC hypersurfaces in the half-space has been proved by Wente \cite{Wente80} via the moving plane method.
The half-ball counterpart has been tackled by Ros-Souam in \cite{RS97}. See also recent developments in this direction \cite{JXZ22, JWXZ22, JWXZ23,  JWXZ23b}.

It is natural to ask for a reasonable Serrin-type overdetermined problem to characterize the capillary spherical caps. In this direction, Pacella-Tralli  \cite{PT20} studies an overdetermined problem in a convex cone, while Guo and the third-named author \cite{GX19} propose a partially overdetermined boundary value problem in a half ball.
Both of these works give a characterization of free boundary spherical caps, that is with $\theta=\frac{\pi}{2}$. See also Magnanini-Poggesi \cite{MP22}, where they establish an integral identity to give a new proof of Guo-Xia's rigidity result \cite{GX19} as well as a quantitative stability analysis;
and also Poggesi's work \cite{Pog22}, where a corresponding integral identity for the overdetermined problem in \cite{PT20} can be found.

In a recent paper \cite{JLXZ23}, 
we give a characterization of capillary spherical caps
 for general $\theta\in(0,\pi)$ in the half-space by a Serrin-type overdetermined problem with certain boundary condition. The corresponding quantitative stability analysis is conducted as well.
The aim of this paper is to  give a characterization of capillary spherical caps
for general $\theta\in(0,\pi)$ in the half-ball case by means of a  Serrin-type overdetermined problem.
To this aim, we
consider the following mixed boundary value problem:
\eq{\label{eq-mixed-halfball}
\begin{cases}
	\bar\Delta f=1&\text{in }\Omega,\\
	f=0&\text{on }\Sigma,\\
	\bar\nabla_\nu f-f=c&\text{on } T,
\end{cases}
}
where $\Omega\subset\mbB_{+}^{n+1}$ is a connected open set,
$\S\coloneqq{\partial \Omega\cap \mbB_{+}^{n+1}}$ and
$T\coloneqq\partial\Omega\setminus\overline\S$ are two smooth hypersurfaces  such that the union of the closure of $\S$ and $T$ is $\p\Om$.
$\nu(x)=x$ is the outward unit normal with respect to $\Om$ along $T$, and $c\in\mathbb{R}$.

\begin{theorem}\label{Thm-rigid-BVP-halfball}
Let $\Omega\subset\mbB_{+}^{n+1}$ 
be a connected open set,
whose boundary $\p\Om$ is made of two smooth parts
$\S\coloneqq{\partial \Omega\cap \mbB_{+}^{n+1}}$,
$T\coloneqq\partial\Omega\setminus\overline\S$.
If the mixed boundary problem $(\ref{eq-mixed-halfball})$, with $\bar{\nabla}_{\nu}f\equiv c_0$ on $\Sigma$, admits a solution $f\in W^{1,\infty}(\Omega)\cap W^{2,2}(\Omega)$ such that $f\leq 0$,
then $c_0>0$ and $\Omega$ must be the form
\eq{
\Omega_{n+1,c,c_0}
=B_{(n+1)c_0}(z)\cap \mbB_{+}^{n+1},
}
where $z\in\mbR^{n+1}$ satisfies
\eq{
\abs{z}
 =\sqrt{1+(n+1)^{2}c_{0}^{2}-2(n+1)c},
}
while the solution $f$ is given by 
\eq{
	f(x)
 =\frac{\abs{x-z}^{2}-(n+1)^{2}c_{0}^{2}}{2(n+1)}.
}
In particular,
$\Sigma$ is a capillary spherical cap that meets $\partial \mbB^{n+1}$ at a constant contact angle $\theta$, characterized by $\cos \theta =-\frac{c}{c_{0}}$.
\end{theorem}

\begin{remark}
\normalfont
\begin{enumerate}[i.]
\item The case $c=0$ has been proposed and handled by Guo-Xia \cite{GX19}. 
\item When $c\leq 0$, the condition $f\leq 0$ is automatically satisfied by maximum principle, see \cite[Proposition 2.3]{GX19}.
\item
If $\Omega$ is such that the contact angle function $\theta$ at each point on $\p\S$ defined as \eqref{defn-contact-func}, which is not essentially constant but satisfies $0<\theta(x)\leq\frac{\pi}{2}$, then the regularity assumption $f\in W^{1,\infty}(\Omega)\cap W^{2,2}(\Omega)$ is valid, thanks to \cite{Liebermann86,Liebermann89}.
For detailed accounts, see \cite[Appendix A]{JXZ22}.
\end{enumerate}
\end{remark}

We shall establish the following crucial integral identity:
\eq{\label{iden-integral}
	&\int_{\Omega}(-Vf)\left\{\abs{\bar\nabla^{2}f}^{2}-\frac{(\bar\De{f})^{2}}{n+1}\right\}\rd x\\
	&=\frac{1}{2}\int_{\Sigma}(f_{\nu}^{2}-(\frac{R}{n+1})^{2})[V(f-\hat g)_{\nu}-V_{\nu}(f-\hat g)]\rd A,
}
where $V(x):=x_{n+1}$,
while $\hat g$ is a quadratic function defined as

\eq{
\hat g(x)=\frac{\abs{x-\hat O}^{2}-\hat R^2}{2(n+1)},
}
here $\hat O\in\mbR^{n+1}$ is such that $\abs{\hat O}^2=1+\hat R^2-2(n+1)c$, $\hat R\in\mbR$.
By virtue of the choice of $\hat O$, it is easy to see that $\hat g$ satisfies
\begin{align}\label{eq:hat-g}
\begin{cases}
	\bar\De \hat g
	=1&\text{in }\Omega,\\
	(\hat g)_\nu-\hat g
	=c&\text{on }T.
\end{cases}
\end{align}

Theorem \ref{Thm-rigid-BVP-halfball} is a direct consequence of \eqref{iden-integral}. In the case $c=0$, this identity has been proved by Magnanini-Poggesi \cite{MP22}.


 \
 
\section{Notations and Preliminaries}\label{Sec-2}
	
\subsection{Notations}\label{Sec-2-1}
Let $\mbR^{n+1}$ be the Euclidean space,  $\left<\cdot,\cdot\right>$ be the scalar product.
Set $\mbB^{n+1}=\{x\in\mbR^{n+1}:\abs{x}<1\}$, $\mbB^{n+1}_+=\{x\in\mbR^{n+1}:\abs{x}<1,x_{n+1}>0\}$, and $E_{n+1}\coloneqq(0,\ldots,0,1)$.
We will use $B_r(z)$ to denote the $(n+1)$-dimensional open ball in $\mbR^{n+1}$ with radius $r$ that is centered at $z$.

Throughout the paper, we consider the domains with the following properties:
$\Omega\subset\mbB_{+}^{n+1}$ is a connected open set with boundary $\partial \Omega=\overline\S\cup \overline T$ where $\Sigma={\partial \Omega\cap \mbB_{+}^{n+1}}$ and $T=\partial \Omega\setminus\overline\S$ are smooth $n$-dimensional manifolds,
and $\Gamma\coloneqq\overline\Sigma\cap \overline T$ is a smooth $(n-1)$-dimensional manifold with $\mathcal{H}^{n-1}(\Gamma)>0$.

We use the following notation for normal vector fields.
$\nu$ denotes the outward unit normal to the smooth part of $\partial\Omega$, namely, $\S$ and $T$.
In particular, we have
\eq{
\nu(x)=x,\quad x\in T.
}
To introduce the boundary contact angle condition, 
we use temporarily the notations $N,\bar N$ to denote the outward unit normal to $\overline\S,\overline T$ with respect to $\Om$ respectively.
Let $\mu$ be the outward unit co-normal to $\Gamma\subset \overline\S$ and $\bar \nu$ be the outward unit co-normal to $\Gamma\subset \overline T$. Under this convention, along $\Gamma$ the bases $\{N,\mu\}$ and $\{\bar \nu,\bar N\}$ have the same orientation in the normal bundle of $\p \S\subset \mfR^{n+1}$.
We define the contact angle function of $\S$ along $\Gamma$, $\theta:\Gamma\rightarrow(0,\pi)$, by
\eq{\label{defn-contact-func}
\left<N(x),-\bar N(x)\right>
=\left<\mu(x),\bar\nu(x)\right>
=\cos\theta(x).
}
In particular, we say that $\S$ meets $\p\mbB^{n+1}$ with a fixed contact angle $\theta$ if along $\Gamma$, \begin{align}\label{munu}
&&\mu=\sin \theta \bar N+\cos\theta \bar \nu,\quad
N=-\cos \theta \bar N+\sin \theta \bar \nu.
\end{align}

We denote by $\bar\nabla,\bar\Delta,\bar\nabla^{2}$ and $\bar{\div}$, the gradient, the Laplacian, the Hessian and the divergence on $\mbR^{n+1}$ respectively.
In the following, for any function $g$, we denote by  $g_{i}$ the derivative of $g$ with respect to $x_i$, by $g_{ij}$
the second-order derivative of $g$ with respect to $x_i$ and $x_j$, and repeated indices will be summed.

\subsection{Preliminaries}

The classical $P$-function related to the solution $f$ of $\bar\Delta f=1$ is denoted by
\eq{
P
 =\frac{1}{2}|\bar{\nabla}f|^{2}-\frac{1}{n+1}f.
}
A direct computation then gives
\eq{
\bar\De P
=\abs{\bar\nabla^{2}f}^{2}-\frac{(\bar\Delta f)^{2}}{n+1}.
}
Using the Cauchy-Schwarz inequality, one finds
\eq{
\bar\De P\geq 0,
}
and equality holds if and only if $\bar\nabla^{2}f$ is proportional to the identity matrix.

Note that given $z\in\mbR^{n+1}$ and an non-empty domain $\Omega\subset\mbB^{n+1}_+$ of the form $\Omega=B_{r}(z)\cap\mbB_{+}^{n+1}$, the related quadratic function $f=\frac{|x-z|^{2}-r^{2}}{2(n+1)}$ solves $\eqref{eq-mixed-halfball}$ for some $c\in \mbR$ with $f_{\nu}\equiv \frac{r}{n+1}$ on $\Sigma$, and a simple computation yields
\eq{\label{eq-R2}
\frac{r}{n+1}\int_{\Sigma}V\rd A
=&\int_{\Sigma}Vf_{\nu}\rd A
=\int_{\partial\Omega}V f_{\nu}\rd A-\int_{T} V f_{\nu}\rd A\\
=&\int_{\partial\Omega}f V_{\nu} \rd A+\int_{\Omega}(V{\bar\Delta} f-f{\bar\Delta} V)\rd x-\int_{T}V f_\nu\rd A\\
=&\int_{T}f V_{\nu} \rd A+\int_{\Omega}V\rd x-\int_{T}V f_\nu\rd A\\
=&\int_{\Omega}V\rd x-c\int_{T}V\rd A,
}
where the third equality holds due to integration by parts,the fourth equality holds because $\bar\De f=1, \bar\Delta V=0$ in $\Omega$, and the last equality holds since $V_\nu=V$ on $T$.
Enlightened by this we set for a general $\Omega$ and a fixed $c\in\mbR$,
\eq{
\mfR(c,\Omega)
\coloneqq(n+1)\frac{\int_{\Omega}V\rd x-c\int_{T}V\rd A}{\int_{\Sigma}V\rd A}
}
which will be referred to as the \textit{reference radius} of $\Omega$ in $\mbB_{+}^{n+1}$.
Such a reference radius generalizes to the present settings \cite[(1.4)]{MP22}.

\section{Proof of Theorem \ref{Thm-rigid-BVP-halfball}}\label{Sec-3}

Regarding the solution to \eqref{eq-mixed-halfball}, we need the following observation that is derived in \cite{GX19}.
\begin{lemma}\label{GX-Prop 2.4}\normalfont(\cite[Proposition 2.4]{GX19})
Let $X$ be a tangent vector field to $T$ and let $f$ be a solution to \eqref{eq-mixed-halfball}.
Then, we have that
\eq{
	\left<X,\bar\nabla^{2}f\nu\right>
 =0\text{ along }T.
}
\end{lemma}

\begin{proof}
Since $f_\nu-f=c$, by differentiating both sides with respect to $X$ and noticing that $\nu(x)=x$ on $T$, we obtain
\eq{
\bar\nabla_{X}f
=\bar\nabla_{X}(\left<\bar\nabla f,\nu\right>)
=\left<X,\bar\nabla^{2}f\nu\right>+\left<\bar\nabla f,\bar\nabla_{X}\nu\right>
=\left<X,\bar\nabla^{2}f\right>+\bar\nabla_{X}f \nu, 
}
the assertion then follows.
\end{proof}
The following is the crucial integral identity:
\begin{proposition}\label{Prop-identity halfball}
Given a domain $\Omega\subset\mbB_{+}^{n+1}$ as in Section \ref{Sec-2-1}.
Let $f\in W^{1,\infty}(\Omega)\cap W^{2,2}(\Omega)$ be a solution to the mixed boundary problem \eqref{eq-mixed-halfball}.
Then for any $R\in \mbR$,
we have that
\eq{\label{int-identity halfball}
	&\int_{\Omega}-Vf\left\{|\bar\nabla^{2} f|^{2}-\frac{(\bar\De f)^{2}}{n+1}\right\}\rd x\\
	&~~=\frac{1}{2}\int_{\Sigma}\left(\abs{\bar\nabla f}^2-\left(\frac{R}{n+1}\right)^2\right)[V(f-\hat g)_{\nu}-V_{\nu}(f-\hat g)]\rd A.
}
\end{proposition}
	
\begin{proof}
We begin by noticing that, since $f\in W^{1,\infty}(\Omega)\cap W^{2,2}(\Omega)$ solves \eqref{eq-mixed-halfball}, all the the integration by parts used in the proof are allowed,
see e.g., \cite[Proposition 2.1]{Pog22}.

Integrating by parts, we get 
\eq{
	\int_{\Omega}-Vf\bar\Delta P\rd x=-\int_{\partial\Omega}VfP_{\nu}\rd A+\int_{\partial\Omega}P(Vf)_{\nu}\rd A-\int_{\Omega}P\bar\Delta(Vf)\rd x.
}
Since $V=x_{n+1}$,
we have that $\bar\Delta V=0$
and hence
$\bar\Delta(Vf)=V+2\left<\bar\nabla V,\bar\nabla f\right>$.
Thus
\eq{
\int_{\Omega}-Vf\bar\Delta P\rd x
 =&-\int_{\partial\Omega}VfP_{\nu}\rd A+\int_{\partial\Omega}P(Vf)_{\nu}\rd A-\int_{\Omega}VP\rd x\\
 &-2\int_{\Omega}P\left<\bar\nabla V,\bar\nabla f\right>\rd A.
}
Computing $P_\nu$,
we get
\eq{\label{eq-V-intergral1}
\int_{\Omega}-Vf\bar\Delta P\rd x
=&-\int_{\partial \Omega}Vff_if_{ij}\nu_{j}\rd A+\frac1{n+1}\int_{\partial\Omega}Vff_{\nu}\rd A\\
&+\int_{\partial\Omega}P(Vf)_{\nu}\rd A-\int_{\Omega}VP\rd x
 -2\int_{\Omega}P\left<\bar\nabla V,\bar\nabla f\right>\rd A.
}

To deal with the term $-\int_{\partial \Omega}Vff_if_{ij}\nu_{j}\rd A$, we construct the following auxiliary function
\eq{
\hat g(x)
\coloneqq\frac{\abs{x-\hat O}^{2}-\hat R^2}{2(n+1)},
}
where $\hat O$ satisfies $\abs{\hat O}^2=1+\hat R^2-2(n+1)c$.
Using the auxiliary function $\hat g$, we can write
\eq{\label{eq-V-difficult term}
&-\int_{\partial \Omega}Vff_if_{ij}\nu_{j}\rd A\\
&~~=-\int_{\partial \Omega}Vf(f_i-(\hat g)_i)f_{ij}\nu_{j}\rd A
-\int_{\partial \Omega}Vf(\hat g)_if_{ij}\nu_{j}\rd A\\
&~~=: {\bf I}+{\bf II}.
}

Applying Lemma \ref{GX-Prop 2.4}, we get that on $T$ the following identities hold:
\eq{
(f_i-(\hat g)_i)f_{ij}\nu_j=&(f_\nu-(\hat g)_\nu)f_{ij}\nu_i\nu_j, \\
(f_{n+1})_\nu=& Vf_{ij}\nu_i\nu_j.
}
Substituting this back into the term ${\bf I}$ 
and using
the fact that $f=0$ on $\Sigma$, we have
\eq{
{\bf I}
=&-\int_{T}Vf(f_i-(\hat g)_i)f_{ij}\nu_{j}\rd A
=-\int_{T}Vf(f_\nu-(\hat g)_\nu)f_{ij}\nu_i\nu_j\rd A\\
=&-\int_{T}f(f_\nu-(\hat g)_\nu)(f_{n+1})_\nu \rd A
=-\int_{T}f(f-\hat g)(f_{n+1})_\nu \rd A\\
=&-\int_{\Omega}\bar\div\left(f(f-\hat g)\bar\nabla f_{n+1}\right)\rd x
=-\int_{\Omega}\<\bar\nabla\left(f(f-\hat g)\right),\bar\nabla f_{n+1}\>\rd x\\
=&-\frac12\int_{\Omega}(2f-\hat g){\left({\abs{\bar\nabla f}}^2\right)}_{n+1}\rd x
+\int_{\Omega}f\<\bar\nabla \hat g,\bar\nabla f_{n+1}\>\rd x,
}
where the fourth equality holds by virtue of the Robin conditions of $f$ and $\hat g$ on $T$ (recall \eqref{eq:hat-g}), the fifth equality holds due to the divergence theorem, and the sixth equality holds since $\bar\Delta f=1$ in $\Omega$ and hence $\bar\De f_{n+1}=0$ in $\Omega$. 
Noting that
\eq{
&-\frac12\int_{\Omega}(2f-\hat g){\left({\abs{\bar\nabla f}}^2\right)}_{n+1}\rd x\\
& ~~=-\frac12\int_{\Omega}\bar\div{\left((2f-\hat g){\abs{\bar\nabla f}}^2 E_{n+1}\right)}\rd x
+\frac12\int_{\Omega}(2f-\hat g)_{n+1}{{\abs{\bar\nabla f}}^2}\rd x\\
&~~=-\frac12\int_{\partial\Omega}(2f-\hat g){\abs{\bar\nabla f}}^2\<E_{n+1},\nu\>\rd A
+\frac12\int_{\Omega}(2f_{n+1}-(\hat g)_{n+1}){{\abs{\bar\nabla f}}^2}\rd x,
}
it follows that
\eq{\label{eq-V-term I}
{\bf I}
=& -\frac12\int_{\partial\Omega}(2f-\hat g){\abs{\bar\nabla f}}^2\<E_{n+1},\nu\>\rd A
+\frac12\int_{\Omega}(2f_{n+1}-(\hat g)_{n+1}){{\abs{\bar\nabla f}}^2}\rd x\\
&+\int_{\Omega}f\<\bar\nabla \hat g,\bar\nabla f_{n+1}\>\rd x.
}

On the other hand, using the divergence theorem, we get
\eq{
{\bf II}
=&-\int_{\partial \Omega}Vf(\hat g)_if_{ij}\nu_{j}\rd A
=-\int_{\Omega} (Vf(\hat g)_if_{ij})_j\rd x\\
=&-\int_{\Omega}f\<\bar\nabla g_1,\bar\nabla f_{n+1}\>\rd x
-\frac12\int_{\Omega}V\<\bar\nabla\hat g,\bar\nabla \abs{\bar\nabla f}^2\>\rd x
-\frac1{n+1}\int_{\Omega}Vf\rd x,
}
where the third equality holds because $\bar\Delta f=1$, $(\hat g)_{ij}=\frac{\delta_{ij}}{n+1}$, and $\bar\nabla V=E_{n+1}$.
Keeping track of the second term and integrating by parts again, we find
\eq{
&-\frac12\int_{\Omega}V\<\bar\nabla g_1,\bar\nabla \abs{\bar\nabla f}^2\>\rd x\\
&~~=-\frac12\int_{\Omega}\bar\div\left(V\abs{\bar\nabla f}^2 \bar\nabla\hat g\right)\rd x
+\frac12\int_{\Omega}\abs{\bar\nabla f}^2\<\bar\nabla V, \bar\nabla\hat g\>\rd x
+\frac12\int_{\Omega}V \abs{\bar\nabla f}^2\bar\Delta\hat g\rd x \\
&~~=-\frac12\int_{\partial\Omega}V\abs{\bar\nabla f}^2 (\hat g)_{\nu}\rd A
+\frac12\int_{\Omega}\abs{\bar\nabla f}^2 (\hat g)_{n+1} \rd x
+\frac12\int_{\Omega}V \abs{\bar\nabla f}^2\rd x.
} 
Hence the term ${\bf II}$ reads
\eq{\label{eq-V-term II}
{\bf II}
=&-\int_{\Omega}f\<\bar\nabla g_1,\bar\nabla f_{n+1}\>\rd x
-\frac1{n+1}\int_{\Omega}Vf\rd x\\
&-\frac12\int_{\partial\Omega}V\abs{\bar\nabla f}^2 (\hat g)_{\nu}\rd A
+\frac12\int_{\Omega}\abs{\bar\nabla f}^2 (\hat g)_{n+1} \rd x
+\frac12\int_{\Omega}V \abs{\bar\nabla f}^2\rd x. 
}
Substituting \eqref{eq-V-term I} and \eqref{eq-V-term II} back into \eqref{eq-V-difficult term}, we obtain
\eq{\label{eq-V-difficult term2}
-\int_{\partial \Omega}Vff_if_{ij}\nu_{j}\rd A
=& -\frac12\int_{\partial\Omega}(2f-\hat g){\abs{\bar\nabla f}}^2\<E_{n+1},\nu\>\rd A\\
&-\frac12\int_{\partial\Omega}V\abs{\bar\nabla f}^2 (\hat g)_{\nu}\rd A
+\int_{\Omega}f_{n+1}{{\abs{\bar\nabla f}}^2}\rd x\\
&-\frac1{n+1}\int_{\Omega}Vf\rd x
+\frac12\int_{\Omega}V \abs{\bar\nabla f}^2\rd x.
}
Combing \eqref{eq-V-intergral1} with \eqref{eq-V-difficult term2}, a direct computation shows
\eq{\label{eq-V-intergral2}
&\int_{\Omega}-Vf\bar\Delta P\rd x\\
&~~=\frac{2}{n+1}\int_{\Omega}ff_{n+1}\rd x
 -\frac12\int_{\partial\Omega}(2f-\hat g){\abs{\bar\nabla f}}^2\<E_{n+1},\nu\>\rd A\\
&~~~-\frac12\int_{\partial\Omega}V\abs{\bar\nabla f}^2 (\hat g)_{\nu}\rd A
+\frac1{n+1}\int_{\partial\Omega}Vff_{\nu}\rd A
+\int_{\partial\Omega}P(Vf)_{\nu}\rd A\\
&~~=\frac{2}{n+1}\int_{\Omega}ff_{n+1}\rd x-\frac1{n+1}\int_{\partial\Omega}f^2
V_\nu
\rd A\\
&~~~+\frac12\int_{\partial\Omega}\abs{\bar\nabla f}^2[V(f-\hat g)_{\nu}-V_{\nu}(f-\hat g)]\rd A\\
&~~=\frac12\int_{\partial\Omega}\abs{\bar\nabla f}^2[V(f-\hat g)_{\nu}-V_{\nu}(f-\hat g)]\rd A,
}
where the second equality holds by  virtue of the fact that $\bar\nabla V=E_{n+1}$, the last equality holds because
 \eq{
\frac{2}{n+1}\int_{\Omega}ff_{n+1}\rd x
=\frac1{n+1}\int_{\Omega}\bar\div\left(f^2 E_{n+1}\right)\rd x
=\frac1{n+1}\int_{\partial\Omega}f^2
V_\nu
\rd A.
}

 Since $V_{\nu}=V$ on $T$, by virtue of the Robin conditions of $f$ and $\hat g$ on $T$ again, we know that
\eq{
V(f-\hat g)_{\nu}-V_{\nu}(f-\hat g)=0 \text{ on } T,
}
and it follows from the divergence theorem that
\eq{
\int_{\partial\Omega}\left[V(f-\hat g)_{\nu}-V_{\nu}(f-\hat g)\right]\rd A
=\int_{\Omega}V\bar\Delta (f-\hat g)\rd x
-\int_{\Omega}(f-\hat g)\bar\Delta V\rd x
=0.
}
{T}hus for any $R\in\mbR$, we have
\eq{
&\int_{\partial\Omega}\abs{\bar\nabla f}^2[V(f-\hat g)_{\nu}-V_{\nu}(f-\hat g)]\rd A\\
&~~=\int_{\partial\Omega}\left(\abs{\bar\nabla f}^2-\left(\frac{R}{n+1}\right)^2\right)[V(f-\hat g)_{\nu}-V_{\nu}(f-\hat g)]\rd A\\
&~~=\int_{\Sigma}\left(\abs{\bar\nabla f}^2-\left(\frac{R}{n+1}\right)^2\right)[V(f-\hat g)_{\nu}-V_{\nu}(f-\hat g)]\rd A,
}
and hence \eqref{eq-V-intergral2} reads
\eq{
\int_{\Omega}-Vf\bar\Delta P\rd x
=&\frac12\int_{\Sigma}\left(\abs{\bar\nabla f}^2-\left(\frac{R}{n+1}\right)^2\right)[V(f-\hat g)_{\nu}-V_{\nu}(f-\hat g)]\rd A.
}
This completes the proof.
\end{proof}
	
Now we can prove our main result.
\begin{proof}[Proof of Theorem \ref{Thm-rigid-BVP-halfball}]
Since $f_{\nu}\equiv c_{0}$ on $\Sigma$, we may repeat the argument \eqref{eq-R2} to see that the constant $c_0$ is in fact characterized by
\eq{
c_{0}
=f_{\nu}
=\frac{\int_{\Omega}V\rd x-c\int_{T}V\rd A}{\int_{\Sigma}V\rd A}. 
}

By assumption $f\leq0$ in $\Om$, so that $f<0$ in $\Omega$ by the strong maximum principle.
Since $V=x_{n+1}>0$ in $\Omega\subset \mbB_{+}^{n+1}$, 
taking the integral identity \eqref{int-identity halfball} (with $R=(n+1)c_{0}$) into consideration,
the Cauchy-Schwarz inequality then implies that $\bar\nabla^{2} f(x)$ is proportional to the Identity matrix for any $x\in\Omega$.
Since $\bar\Delta f=1$, we deduce that $f$ must be of the form
\eq{
f(x)
=\frac{|x-z|^{2}}{2(n+1)}+C, 
}
for some $C\in\mbR$.
By the connectedness of $\Omega$ and the fact that $f=0$ on $\Sigma$, we conclude that $\Sigma$ is spherical and hence $\Omega$ is the part of a ball centered at $z\in R_{+}^{n+1}$ with radius $r$ given by
\eq{
r
=|x-z|
=\sqrt{-2(n+1)C}\text{ on }\Sigma. 
}

Since $f_{\nu}=c_{0}$ on $\Sigma$, we find at any $x\in\Sigma$, 
\eq{
c_{0}
=f_{\nu}(x)
=\frac{|x-z|}{n+1}>0, 
}
implying that $r=(n+1)c_{0}$, $C=-\frac{r^2}{2(n+1)}=-\frac{n+1}{2}c_0^2$.
On the other hand, for $x\in T\subset\partial\mbB_+^{n+1}$, we have
 \eq{
c=f_{\bar N}-f=\<\frac{x-z}{n+1},x\>-\frac{|x-z|^{2}}{2(n+1)}-C=\frac{1-\abs{z}^2}{2(n+1)}+\frac{n+1}{2}c_0^2,
}
it follows that $\abs{z}
 =\sqrt{1+(n+1)^{2}c_{0}^{2}-2(n+1)c}$.
Moreover, since $\Sigma$ is spherical, the contact angle of $\Sigma$ and $\partial \mbB_{+}^{n+1}$ must be a constant, say $\theta\in(0,\pi)$.
By the cosine theorem, we have 
\eq{
\cos\theta=-\cos(\pi-\theta)=-\frac{1+r^2-\abs{z}^2}{2r}=-\frac{c}{c_0}.
}
Therefore, we conclude that $\Omega$ and $f$ have the desired properties, which completes the proof.

\end{proof}

\noindent{\bf 
Acknowledgments.}

X.J. is supported by Natural Science Foundation of Jiangsu Province, China (Grant No. BK20241258).

Z.L. is supported by CSC (No.202206310081)

C.X. is supported by the National Natural Science Foundation of China (Grant No. 12271449, 12126102).

We would like to thank the referee for careful reading and valuable suggestions which helped improve the exposition of the paper.


\bibliography{reference.bib}

\end{document}